\documentclass[11pt,twoside,a4paper]{amsart}
\usepackage{amssymb}
\date{\today}

\title 
[A note on residue currents of integrally
closed ideals]
{A note on the singularities of residue currents of integrally closed
  ideals}

\author{Elizabeth Wulcan}
\thanks{The author was partially supported by the Swedish Research Council.} 
\address{Department of Mathematical Sciences\\Chalmers University of
  Technology and University of Gothenburg\\SE-412 96 
Gothenburg\\SWEDEN}
\email{wulcan@chalmers.se}

\usepackage{geometry}
\usepackage{xcolor}

\newtheorem{thm}{Theorem}[section]

\theoremstyle{definition}

\theoremstyle{remark}

\newtheorem{preremark}[thm]{Remark}
\newtheorem{preex}[thm]{Example}

\newenvironment{ex}{\begin{preex}}{\qed\end{preex}}

\def\K{{K}}
\def\a{{\mathfrak a}}

\DeclareMathOperator{\ord}{ord}
\DeclareMathOperator{\sign}{sgn}
\DeclareMathOperator{\np}{NP}
\def\1{\mathbf 1}
\def\m{{\mathfrak m}}
\def\Np{{\mathbf N}}

\def\F{{D}}

\newcommand{\cone}{\mathcal C}

\newcommand{\Image}{{\text{Im}\,}}

\newcommand{\C}{\mathbf{C}}

\newcommand{\Z}{\mathbf{Z}}

\newcommand{\dbar}{\bar{\partial}}

\newcommand{\R}{\mathbf{R}}

\newcommand{\Q}{\mathbf{Q}}

\newcommand{\Ok}{\mathcal{O}}

\newcommand{\codim}{{\text{codim}\,}}

\newcommand{\Hom}{{\text{Hom}\,}}
\newcommand{\End}{{\text{End}\,}}

\def\newop#1{\expandafter\def\csname #1\endcsname{\mathop{\rm #1}\nolimits}}
\newop{span}

\numberwithin{equation}{section}

\begin{document}
\nocite{*}
\bibliographystyle{plain}

\begin{abstract}
Given a free resolution of an ideal $\a$ of holomorpic functions there
is an associated residue current $R$ that coincides with the classical
Coleff-Herrera product if $\a$ is a complete intersection ideal and whose
annihilator ideal equals $\a$. In the case when $\a$ is an Artinian
monomial ideal, we show that the singularities of $R$ are small in a
certain sense if and only if $\a$ is integrally closed. 
\end{abstract}


\maketitle


\section{Introduction}

Given (a germ of) a holomorphic function $f$ at $0\in \C^n$, 
Herrera and Lieberman, \cite{HL}, proved that one can define the \emph{principal
value} current 
\begin{equation}\label{borjan}
\frac{1}{f}
.\xi:=\lim_{\epsilon\to 0}\int_{|f|^2>\epsilon}\frac{\xi}{f},
\end{equation}
for test forms $\xi$.
It follows that 
$\dbar(1/f)$ is a current with support on the variety
of $f$; such a current is called a \emph{residue current}. 
The \emph{duality principle} asserts that 
a holomorphic
germ $g$ is in the ideal 
generated by $f$ if and only if $g\dbar
(1/f)=0$.

Given a (locally) free resolution 
\begin{equation}\label{foten}
    0 \to E_{N} \stackrel{\varphi_{N}}{\longrightarrow} E_{N-1} \to \cdots \to
    E_1 \stackrel{\varphi_1}{\longrightarrow} E_{0}\to 0 
\end{equation}
of a general ideal (sheaf) $\a$, in \cite{AW} with Andersson we defined a vector (bundle) valued 
residue current $R$ with support on the variety of $\a$ 
that satisfies the duality principle for $\a$, cf.\ Section ~\ref{nya} below. If
  $\a$ is Cohen-Macaulay, then $R$ is essentially independent of the
  resolution. In particular, if 
 $(E,\varphi)$ is the Koszul complex of a minimal set of generators 
$f_1,\ldots, f_p$ of  
a complete intersection ideal, 
then $R$
coincides with the classical Coleff-Herrera product, \cite{CH},  
\begin{equation}\label{coleff}
\dbar\frac{1}{f_p}\wedge\cdots\wedge\dbar\frac{1}{f_1}. 
\end{equation}
By means of these residue currents we were able to extend several results
previously known for complete intersections.
These currents have also turned out to be 
particularly useful for analysis on singular spaces;
for example, 
they have been used to obtain new results on the $\dbar$-equation,
\cite{AS}, 
and new global versions of the classical
Brian\c{c}on-Skoda theorem, \cite{AWSemester}, on singular spaces.

In view of the duality principle the residue current $R$ can be thought of as
a current representing the ideal $\a$; this idea is central to many
applications of residue currents, including the ones mentioned above. 
Various properties of the ideal $\a$ are reflected in the residue
current $R$. 
For example, 
$R$ has a natural geometric decomposition corresponding to a primary
decomposition of $\a$, see \cite{AW2}, and the 
fundamental cycle of $\a$ admits a natural
representation in terms of $R$ that generalizes the classical
Poincar\'e-Lelong formula, see \cite{LW}.

In this note we study the singularities of $R$ and show that, for a 
monomial ideal $\a$, they are small in a certain
sense if and only if $\a$ is 
integrally closed.
For simplicity we will work in a
  local setting; let $\Ok_0^n$ be the ring of germs of holomorphic functions at
$0\in\C^n$ and let $\a$ be an ideal in $\Ok_0^n$. 
Recall that $g\in \Ok_0^n$ 
is in the
\emph{integral closure} $\overline \a$ of $\a$ if $|g|\leq 
C |f|$, 
where $C$ is a constant and $f$ is a set of generators $f_1,\ldots,
f_m\in\Ok_0^n$ of $\a$, or equivalently if $g$
satisfies a monic equation $g^q + h_1 g^{q-1}+\cdots + h_q=0$, where
$h_k\in \a^k$. If $\overline \a=\a$, then $\a$ is said to be
\emph{integrally closed}. 
Assume that $\pi:\widetilde X\to (\C^n,0)$ is a log-resolution of
  $\a$, i.e., $\widetilde X$ is a complex manifold, $\pi$ is a
  biholomorphism outside the variety of $\a$, 
and $\a\cdot \Ok_{\widetilde
    X}=\Ok_{\widetilde
    X}(-D)$, where $D=\sum_{i=1}^Nr_i D_i$ is
  an  effective divisor with simple normal crossings
  support. 
Then $\overline \a=\pi_* (\Ok_{\widetilde X}(-D))$, which means that
$g\in\Ok_0^n$ is in $\overline \a$ if and only if 
$\ord_{D_i}(g)\geq r_i$ for each $i$, where $\ord_{D_i}$ denotes the divisorial valuation
defined by the prime divisor $D_i$.

If $\pi: \widetilde X\to (\C^n,0)$ is a common log-resolution
of $\a$ and the Fitting 
ideals of $\a$, i.e., the ideals generated by the 
minors of optimal rank of the $\varphi_k$ in \eqref{foten}, 
then there is a section $\sigma$ of a line bundle $L=\Ok_{\widetilde X}(-F)$
over $\widetilde X$ and a current $\widetilde R$ on $\widetilde X$ such
that
\begin{equation}\label{bara}
\widetilde R\wedge \pi^* dz= \eta\wedge \dbar \frac{1}{\sigma},
\end{equation} 
where $dz=dz_1\wedge \cdots \wedge dz_n$ and $\eta$ is a vector
(bundle) valued smooth form with values in $L$, 
such that $\pi_*\widetilde R=R$, see \cite[Section~2]{AW} and Section~\ref{nya} below. 
The observation that residue currents in this way can be seen as
pushforwards of residue currents of principal ideal sheaves  
is crucial
for many applications of residue currents, cf.\ Section ~\ref{asm} below.

Assume that 
\begin{equation}\label{jul}
\sigma=\sigma_1^{a_1}\cdots \sigma_N^{a_N},
\end{equation} 
where $\sigma_i$ are holomorphic sections of line bundles $\Ok(-D_i)$
defining the prime divisors $\F_i$ of $F=\sum_{i=1}^N a_i D_i$.  
We are interested in the exponents
$a_i$. Naively one could hope that one could choose $a_i$ as 
$r_i=\ord_{\F_i}(\a)$. However, this can only be true if $\a$ is integrally closed. 
Indeed, assume that $R=\pi_* \widetilde R$, where $\widetilde R$
satisfies \eqref{bara} with $\sigma$ given by
\eqref{jul} with $a_i \leq r_i$. 
Take $g\in \overline \a$; then 
$\ord_{\F_i}(g)\geq r_i$ for each $i$ and thus $\pi^* g = \sigma g'$,
where $g'$ is a holomorphic section of $L^{-1}$. 
Therefore, by the duality principle,
$\pi^* g\dbar (1/\sigma)=0$ and so 
$\pi^* g\widetilde R=0$, which implies that $gR=0$, and hence 
$g\in \a$. 
To conclude, if we can choose $a_i\leq r_i$ 
for each $i$, then
$\a$ is integrally closed.

We are interested in whether the converse holds, i.e., if $\a$
is integrally closed, is it then 
always possible to find an $\widetilde R$ as above with $\sigma$ given
by \eqref{jul} with  
$a_i\leq r_i$? 
In this note we answer this question affirmatively when 
$R$ is the residue current associated with a cellular resolution,
introduced by Bayer-Sturmfels \cite{BS}, see Section ~\ref{zlatan}
below, of an Artinian, i.e., $0$-dimensional,
monomial ideal, and when we moreover allow $\eta$ to be semi-meromorphic, i.e.,
of the form $(1/f)\omega$, where $\omega$ is smooth and $f$ is holomorphic. 
For the
definition of the 
product $\eta\wedge \mu$, where $\eta$ is a semi-meromorphic form and $\mu$ is a
residue current, or more generally a so-called pseudomeromorphic current, see Section ~\ref{asm} below. Multiplication from the
left by $\eta$ does not increase the singularities 
in the sense that if $g$ is a holomorphic function such that $g\mu=0$, then $g\eta\wedge \mu=0$.

\begin{thm}\label{huvud} 
Let $M\subset\Ok_0^n$ be an integrally closed Artinian monomial ideal and let $R$ be the residue
current associated with a cellular resolution of $M$. 
Then there is a log-resolution $\pi:\widetilde X\to(\C^n,0)$ of $M$,
such that $M\cdot \Ok_{\widetilde X}=\Ok_{\widetilde X} (-D)$, where 
$D=\sum_{i=1}^N r_i D_i$,
and a current 
$\widetilde R$ on $\widetilde X$ such that $\pi_*\widetilde R=R$ and
\begin{equation}\label{underbara}
\widetilde R\wedge \pi^* dz= \eta\wedge \dbar
\frac{1}{\sigma_1^{r_1}\cdots \sigma_N^{r_N}},
\end{equation} 
where $\sigma_i$ is a holomorphic section defining $D_i$ and $\eta$ is a
semi-meromorphic form.
\end{thm}

The proof uses explicit descriptions of residue currents of
monomial ideals, \cite{W}, as well as so-called
Bochner-Martinelli residue currents, \cite{JW}, cf.\ Sections
~\ref{zlatan} and 
~\ref{juvel} below, and it should be possible to extend to general, not
necessarily Artinian, monomial ideals.
There is a brief discussion of this and other aspects of our result at  
the end of Section ~\ref{navidad}.

\smallskip 
\noindent
\textbf{Acknowledgment.} I would like to thank Mats Andersson for
valuable discussions on the topic of this paper. I am also very grateful to
the referee for the careful reading, for pointing out some obscurities and
mistakes in
a previous version of this paper,
and for several useful comments and suggestions.

\section{Preliminaries}

\subsection{Pseudomeromorphic  currents}\label{asm}

To get a coherent approach to principal value and residue currents, 
in \cite{AW2} with Andersson we introduced the sheaf of \emph{pseudomeromorphic} currents which
essentially are push-forwards of tensor products of principal value
and residue currents times smooth forms, like 
\[
\frac{1}{s_2^{b_2}\cdots s_m^{b_m}}
\omega \wedge \dbar 
\frac{1}{s_1^{b_1}},
\]
where $s_1,\ldots, s_m$ are (local) coordinates in some $\C^m$ and
$\omega$ is a smooth form. 
Principal value currents and the   
residue currents mentioned in this paper  
are typical examples of pseudomeromorphic currents.

Pseudomeromorphic currents have a geometric nature similar to positive
closed currents. 
For example, the {\it dimension principle} states that if the pseudomeromorphic current
$\mu$ has bidegree $(*,p)$ and support on a variety of codimension larger than 
$p$, then $\mu$ vanishes.  Moreover 
if $\mu$ is a pseudomeromorphic current and $\1_V$ is the characteristic
function of an analytic variety $V$, then 
the product $\1_V\mu$, defined through a suitable regularization, is a well-defined pseudomeromorphic current with support on
$V$, 
see \cite[Proposition~2.2]{AW2}.

A current of the form $(1/f)\omega$, where $f$ is a holomorphic section
of a line bundle $L\to X$ and $\omega$ is a smooth form with values in
$L$ is said to be \emph{semi-meromorphic}. 
If $\eta$ is a semi-meromorphic form, or more generally the
push-forward under a modification of a semi-meromorphic form, and
$\mu$ is a 
pseudomeromorphic current, there is a unique pseudomeromorphic current
$\eta\wedge\mu$ that coincides with the usual product 
where $\eta$ is smooth and such that $\1_{ZSS(\eta)} \eta\wedge \mu=0$,
where $ZSS(\eta)$ is the smallest analytic set containing the set where $\eta$ is not
smooth, see, e.g., \cite[Section~4.2]{AW3}. 
If $h$ is a holomorphic tuple such that $\{h=0\}=ZSS(\eta)$ and $\chi(t)$ is (a smooth approximand of) the characteristic function of the
interval $[1,\infty)$, 
then 
\begin{equation*}
\eta\wedge \mu=\lim_{\epsilon\to 0}\chi(|h|^2/\epsilon)
\eta\wedge\mu.  
\end{equation*}
It follows that for $c>0$  
\begin{equation}\label{pippi}
\frac{1}{s_i^c}~ 
\frac{1}{s_i^b}=  
\frac{1}{s_i^{b+c}}, \quad 
\frac{1}{s_i^c}~ 
\dbar\frac{1}{s_i^b}=  0.
\end{equation} 
For further reference, in view of \eqref{borjan}, note  that 
\begin{equation}\label{pippi2}
s_i^c~ 
\frac{1}{s_i^b}=  
\frac{1}{s_i^{b-c}}, \quad 
s_i^c~ 
\dbar\frac{1}{s_i^{b}}= \dbar\frac{1}{s_i^{b-c}}.
\end{equation} 

\begin{ex}\label{kronan}
Assume that $s_1,\ldots, s_n$ are (local) coordinates in $\C^n$. 
If $D=\{s_1=0\}$, by the dimension
principle, 
\[
\1_D ~\dbar \frac{1}{s_1^{b_1}\cdots s_n^{b_n}}= 
\frac{1}{s_2^{b_2}\cdots s_n^{b_n}} ~ \dbar \frac{1}{s_1^{b_1}}. 
\]
In view of \eqref{pippi} and \eqref{pippi2} 
it follows that for $c_1\geq b_1$ and any choices of $c_2,\ldots,
c_n$, 
\[
\1_D ~\dbar \frac{1}{s_1^{b_1}\cdots s_n^{b_n}}=
\1_D ~ s_1^{c_1-b_1}\cdots s_n^{c_n-b_n} 
~\dbar \frac{1}{s_1^{c_1}\cdots s_n^{c_n}},  
\]
where the factor $s_i^{c_i-b_i}$ should be understood as a principal
value if $c_i<b_i$. 
\end{ex}

\subsection{Residue currents from complexes of vector bundles}\label{nya}
Let 
\begin{equation}\label{komplexet}
    0 \to E_{N} \stackrel{\varphi_{N}}{\longrightarrow} E_{N-1} \to \cdots \to
    E_1 \stackrel{\varphi_1}{\longrightarrow} E_{0}\to 0 
\end{equation}
be a complex of Hermitian vector bundles over a complex manifold $X$ of
dimension $n$ 
that is exact outside a variety $Z\subset X$. In \cite{AW} with
Andersson we constructed an $\End (\oplus E_k)$-valued 
residue current
$R$ with support on $Z$, 
that in some sense measures the exactness of the associated sheaf complex 
\begin{equation}\label{sektioner}
    0 \to \Ok(E_{N}) \stackrel{\varphi_{N}}{\longrightarrow} \Ok(E_{N-1}) \to \cdots \to
    \Ok(E_1) \stackrel{\varphi_1}{\longrightarrow} \Ok(E_{0})\to 0 
\end{equation}
of holomorpic sections. 
If \eqref{sektioner} is exact, then $R$ satisfies the
duality principle, which means that if $\xi$ is a section of $E_0$
that is generically in the image of $\varphi_1$, 
then $R\xi=0$ if and only if $\xi\in
\Image \varphi_1$; in particular, if \eqref{sektioner} is a free resolution of an
ideal $\a\subset \Ok_0^n$, then the
\emph{annihilator ideal} of $R$, i.e., the ideal of germs of
holomorphic functions $g$ such that $gR=0$,  
equals $\a$. 
Moreover, then $R$ is of the form $R=\sum R_k$, where
$R_k$ is a $\Hom(E_0,E_k)$-valued pseudomeromorphic current of bidegree $(0,k)$. 
Note that $R_k$ vanishes for $k < \codim Z$ by the dimension
principle, and for $k>n$ for degree reasons. In particular, if \eqref{sektioner} is a free resolution of
an Artinian monomial ideal in $\Ok_0^n$, then $R=R_n$. 

Let $\rho_k$ be the optimal rank of $\varphi_k$, and let
$\pi:\widetilde X\to X$ be a common log-resolution of the ideal
sheaves generated by the $\rho_k$-minors of the $\varphi_k$, 
i.e., such that the pullback of the section $\det^{\rho_k} \varphi_k$ of
$\Lambda^{\rho_k} E_k^* \otimes \Lambda^{\rho_k} E_{k-1}$ is of the form $t_k\rho_k'$,
where $t_k$ is a section of some line bundle $L_k$ and $\varphi_k'$ is
a nonvanishing section of $L_k^{-1}\otimes \Lambda^{\rho_k}
\pi^* E_k^* \otimes \Lambda^{\rho_k}
\pi^* E_{k-1}$. 
It was proved in \cite[Section~2]{AW} that there is a current
$\widetilde R$ on $\widetilde X$ such that $\pi_* \widetilde R = R$
and $\widetilde R = \omega \wedge \dbar (1/\sigma)$, where $\omega$ is smooth and $\sigma=t_1\cdots
t_{\min (n,N)}$. 
The form $\omega$ may vanish along the divisor $F$ of $\sigma$, and
thus in general it may be possible to find a $\sigma$ that vanishes to lower
order along $F$ than $t_1\cdots t_{\min (n,N)}$, cf.\ Example ~\ref{bruno}
below.

\subsection{Monomial ideals and cellular resolutions}\label{zlatan}
Let us briefly recall the construction of cellular resolutions due to
Bayer-Sturmfels, \cite{BS}. Let $M$ be a monomial ideal in the
polynomial ring $S:=\C[z_1,\ldots,z_n]$, i.e.,
$M$ is generated by monomials $m_1, \ldots, m_r$ in $S$. 
Moreover, let $\K$ be an oriented polyhedral cell complex, where the vertices are
labeled by the generators $m_i$ and the face $\tau$ of $\K$ is 
labeled by the least common multiple $m_\tau$ of the labels $m_i$
of the vertices of $\tau$. 
Then with $\K$ there is an associated graded complex of
$S$-modules. For $k=0,\ldots, \dim \K+1$, let $A_k$ be the free
$S$-module with one generator $e_\tau$ in degree $m_\tau$ for each
$\tau\in K_k$, where $\K_k$ denotes the faces of $\K$ of dimension $k-1$ ($K_0$ should be interpreted as $\{\emptyset\}$) and let 
$\varphi_k:A_k\to A_{k-1}$ be defined by 
\begin{equation}\label{knape}
\varphi_k: e_\tau  \mapsto  
\sum_{\tau'\subset \tau} \sign(\tau',\tau)~\frac{m_\tau}{m_{\tau'}}~ e_{\tau'},
\end{equation}
where $\sign(\tau',\tau)$ is a sign coming from the orientation of
$\K$. 
Now the complex $(A, \varphi)$ is exact precisely if the labeled
polyhedral cell complex $\K$ satisfies a certain acyclicity
condition, see \cite[Proposition~1.2]{BS}. We then say that the complex $(A, \varphi)$ is a
\emph{cellular resolution} of $M$. 
Any monomial ideal admits a cellular resolution,
cf.\ \cite[Proposition~1.5]{BS}.

Let $M$ denote also the monomial ideal of germs of holomorphic
functions at $0\in \C^n_{z_1,\ldots,
  z_n}$ generated by the monomials
$m_1,\ldots, m_r$. 
Since $\Ok_0^n$ is flat over $S$,  
$(A, \varphi)$ induces a
free resolution of $M\subset \Ok_0^n$. More precisely, for $k=0,\ldots, N=\dim K+1$, let $E_k$ be
a trivial bundle over (a neighborhood of $0$ in) ~$\C^n$ 
with a global frame $\{e_\tau\}_{\tau\in \K_k}$,
endowed with the trivial metric, and where the
differential $\varphi_k$ is given by ~\eqref{knape}. 
Then \eqref{sektioner} 
is exact if $(A,\varphi)$ is. 
We will think of monomial ideals sometimes as ideals in $S$, sometimes as ideals in $\Ok_0^n$, and sometimes as ideals in
the ring of entire functions in $\C^n$.

In \cite{W} we computed the residue current $R$ of 
a cellular resolution of a monomial ideal $M$. 
Note that if $M$ is Artinian, then $R=R_n$ is of the form $R=\sum R_\tau e_\tau\otimes
  e_\emptyset^*$, i.e., with one component for each $\tau\in K_n$. 
Proposition~3.1 in that paper asserts that if 
$z^\alpha:=z_1^{\alpha_1}\cdots z_n^{\alpha_n}$ is the label of
$\tau$, then 
$R_\tau=c_\tau R_\alpha$, where $c_\tau\in\C$ and 
\begin{equation}\label{sara}
  R_\alpha= \dbar\frac{1}{z_n^{\alpha_n}}\wedge\cdots \wedge \dbar \frac{1}{z_1^{\alpha_1}}.
\end{equation}

\subsection{Toric log-resolutions}\label{toric}
For an (Artinian) monomial ideal $M$ in $\C^n$ it is possible to find a
log-resolution $\pi:\widetilde X\to \C^n$ where $\widetilde X$ is a toric variety. 
Let us briefly recall this construction, which can be found, e.g., in
\cite[p.\ 82]{BGVY}. 
 For a general reference on toric varieties, see,
e.g., \cite{Fu}. 
A \emph{(rational strongly convex) cone} in $\R^n$ is a set of the
form $\cone=\sum\R_+v_i$, where $v_i$ are in the lattice $\Z^n$, that
contains no line; 
here $\R_+$ denotes the non-negative real
numbers. A cone is \emph{regular} if the $v_i$ can be chosen as part
of a basis for the lattice $\Z^n$. A \emph{fan} $\Delta$ is a finite collection
of cones such that all faces and intersections of
cones in $\Delta$ are in $\Delta$; $\Delta$ is \emph{regular} if all
cones are regular. 
A regular fan $\Delta$ determines a smooth toric variety $X(\Delta)$, 
obtained by patching together affine toric varieties 
corresponding to the cones in $\Delta$.

Assume that $M$ is an Artinian monomial ideal in $\C^n_{z_1,\ldots,
  z_n}$. 
Recall that the \emph{Newton polyhedron} $\np(M)$ of $M$ is
defined as the convex hull in $\R^n$ of the exponents of monomials in
$M$. 
Let $\mathcal S(M)$ be the collection of cones of the form
$\cone=\R_+\rho\subset \R^n_+$, 
where $\rho$ is a normal vector of a compact facet (face of maximal
dimension) 
of $\np (M)$. 
Let $\Delta$ be a regular fan that contains $\mathcal S(M)$ and such
that the \emph{support}, i.e., the union of all cones in $\Delta$,
equals $\R^n_+$.
The cones in $\mathcal S(M)$ determine a fan with support $\R^n_+$ and
by refining this is always possible to find such a $\Delta$. 
Then $\pi:X(\Delta)\to \C^n$ is a log-resolution of
$M$. The prime divisors $D_i$ of the exceptional divisor correspond to 
one-dimensional cones $\cone_i=\R_+ \rho^i$ in
$\Delta$ and $\ord_{D_i}$ are monomial valuations 
(i.e., determined by their values on $z_1,\ldots, z_n$). More
precisely, if $\rho=(\rho_1,\ldots, \rho_n)$ is the first non-zero
lattice point met along 
$\cone_i$, then $\ord_{D_i}$ is the monomial valuation $\ord_\rho(z_1^{a_1}\cdots z_n^{a_n}):=\rho_1 a_1+\cdots
+\rho_n a_n$.

\subsection{Rees valuations}\label{svara}
Given a non-zero ideal $\a\subset \Ok_0^n$, let $\nu:X^+\to (\C^n,0)$
be the normalized blow-up of $\a$ and let $D=\sum r_i D_i$ be the
exceptional divisor, such that $\a\cdot
  \Ok_{X^+}=\Ok_{X^+}(-D)$.
The divisorial valuations $\ord_{D_i}$ are called the \emph{Rees
  valuations} of $\a$, see, e.g., \cite[Section~9.6.A]{LazII}. 
Then $\overline \a=\nu_* (\Ok_{X^+}(-D))$, i.e., $g\in \overline \a$
if and only if $\ord_{D_i}(g)\geq \ord_{D_i}(\a)$ for each $i$. 
If $\pi:\widetilde X\to (\C^n, 0)$ is any log-resolution of $\a$ 
(and thus factors through the normalized blow-up) with exceptional
divisor $D=\sum r_i D_i$, following \cite{JW}
we say that the prime divisor $D_i$ 
is a \emph{Rees divisor} if $\ord_{D_i}$ is a Rees valuation.

Let $M$ be an Artinian monomial ideal (at $0$) in $\C^n_{z_1,\ldots,
  z_n}$. 
Then the Rees valuations are monomial and in one-to-one correspondence with the
compact facets of $\np(M)$. 
We say that the normal vector $\rho$ of a facet $\tau$ is 
\emph{primitive} if it has minimal non-negative entries, i.e., if $\rho$ is the first lattice point met along the cone
$\R_+\rho\subset\R^n_+$. If $\rho$ is a primitive normal vector of a
compact facet $\tau$, then the Rees valuation
corresponding to $\tau$ is the monomial valuation $\ord_\rho$, see, e.g., \cite[Theorem~10.3.5]{HS} and
\cite[p.\ 82]{BGVY}. 
It follows that in the toric log-resolution $\pi : X(\Delta)\to \C^n$
in the previous section, $D_i$ is a Rees
divisor of $M$ if
and only if the corresponding cone $\cone_i$ is in 
$\mathcal S(M)$.

\begin{ex}\label{morgon}
Given $\beta = (\beta_1,\ldots, \beta_n)\in \Np^n$, we let
$\m^\beta$ 
denote the Artinian monomial complete intersection ideal generated 
by $z_1^{\beta_1}, \ldots, z_n^{\beta_n}$. Then $\np(\m^\beta)$ has a
unique compact facet, namely the simplex $\tau$ with vertices 
$(\beta_1,0\ldots, 0), (0,\beta_2,0,\ldots, 0),$ $\ldots, (0,\ldots, 0,
\beta_n)$. 
Let $\rho_j = \beta_1\cdots \beta_{j-1}\beta_{j+1}\cdots \beta_n$;
then $\rho=(\rho_1,\ldots, \rho_n)$ is a normal vector of $\tau$ and
thus the unique Rees valuation of $\m^\beta$ is of the form
$r\ord_\rho$ for some $r\in \Q$. 
Note that $\ord_\rho(z_i^{\beta_i})=\ord_\rho(\m^\beta)$ for all ~$i$. 
\end{ex}

\subsection{Bochner-Martinelli residue currents}\label{juvel} 
Let $f = (f_1,\ldots,f_p)$ be a tuple (of germs) of holomorphic
functions at $0\in \C^n$ 
and let \eqref{komplexet} 
be the Koszul complex of $f$, i.e., consider
$f$ as a section $f=\sum f_j e_j^*$ of a trivial rank $p$ bundle $E^*$
over (a neighborhood of $0$ in) $\C^n$ 
with a frame $e_1^*,\ldots, e_p^*$
, let $E_j=\bigwedge^j E$, where $E$
is the dual bundle of $E^*$, and let
$\varphi_k=\delta_f$ be contraction with $f$.
Assume that the complex is equipped with the trivial metric. 
Then the coefficients of the associated
residue current are the so-called \emph{Bochner-Martinelli residue
  currents} introduced by Passare, Tsikh, and Yger, \cite{PTY}. In particular, if $f_1,\ldots,
f_p$ are minimal generators of a complete intersection ideal, then the
only nonvanishing coefficient of $R=R_p$ equals the Coleff-Herrera
product \eqref{coleff},  
see \cite[Theorem~4.1]{PTY} and \cite[Theorem~1.7]{A}.

In \cite{JW}, together with Jonsson, we gave a geometric description
of the residue current $R$ in this case in terms of the Rees
valuations of the ideal $\a=\a(f)$ generated by $f$.  
It is proved in Section~4 in that paper that 
if $\pi:\widetilde X\to
(\C^n,0)$ is a log-resolution of $\a$, then there is a current
$\widetilde R$ such that $\pi_*\widetilde R=R$ and
$\widetilde R$ has support on the Rees divisors of $\a$. 
Moreover if $D=\sum_{i=1}^N r_i D_i$ is the exceptional divisor of $\pi$, then 
\begin{equation}\label{gaspa}
\widetilde R = \omega \wedge \dbar \frac{1}{\sigma_1^{n r_1}\cdots \sigma_N^{n r_N}},
\end{equation}
where 
$\sigma_ i$ is a holomorphic section defining $D_i$ and $\omega$ is a
smooth form.

\begin{ex}\label{bruno} 
Let $\a_\ell=(z_1^\ell,\ldots, z_n^\ell)\subset \Ok_0^n$, and let \eqref{komplexet} be the
Koszul complex of $(z_1^\ell,\ldots,
z_n^\ell)$. Then the $\rho_k$-minors of the $\varphi_k$ are monomials of degree
$\rho_k\ell$. It follows that the blow-up 
of $\C^n$ at $0$ is a common log-resolution of $\a_\ell$ and the ideals generated by the
$\rho_k$-minors of the $\varphi_k$. Let $D=\{\sigma_1=0\}$ denote the
exceptional (prime) divisor. Then $\ord_D(z_i)=1$ for each $i$ and
$\ord_D(dz)=n-1$. It follows that the section $t_k$ from Section
\ref{nya} is of the form $t_k=\sigma_1^{\rho_k \ell}$, so that
according to Section \ref{nya} there
is an $\widetilde R$ that satisfies \eqref{bara} with $\sigma=\sigma_1^{(\rho_1+\cdots +
  \rho_n)\ell - n+1}$ and where $\eta$ is smooth. 
However, noting that $\ord_D(\a_\ell)=\ell$, in view of 
  \eqref{gaspa}, we can, in fact, choose $\widetilde R$ with
  $\sigma=\sigma_1^{(n-1)\ell+1}$.  
\end{ex}

\section{Proof of Theorem ~\ref{huvud}}\label{navidad} 
Theorem ~\ref{huvud} is a direct consequence of the following slightly more precisely
formulated result.

\begin{thm}\label{knaochta}
Let $M\subset \Ok_0^n$ be an integrally closed Artinian monomial ideal and let
$R$ be the residue current associated with a cellular resolution of
$M$ corresponding to the labeled polyhedral cell complex
$\K$.  
Then there is a log-resolution $\pi:\widetilde X\to(\C^n,0)$ of $M$
and a current $\widetilde R$ on $\widetilde X$ with support on the Rees divisors of
$M$ such that $\pi_*\widetilde R= R$ and $\widetilde R\wedge \pi^* dz$
is of the form \eqref{underbara}, where $D=\sum r_i D_i$, 
$\sigma_i$, and $\eta$ are as in Theorem ~\ref{huvud}. 
More precisely, for each $\tau\in \K_n$, there is a current $\widetilde
R_\tau$ on $\widetilde X$ and a Rees divisor $D_\tau$ such that $\widetilde
R_\tau$ has support on $D_\tau$, 
$\pi_*\widetilde R_\tau= R_\tau$, and
\begin{equation*}
\widetilde R_\tau\wedge \pi^* dz= \eta_\tau\wedge \dbar
\frac{1}{\sigma_1^{r_1}\cdots \sigma_N^{r_N}},
\end{equation*} 
where $\eta_\tau$ is a semi-meromorphic form. 
\end{thm}

\begin{proof}
Let $\pi:\widetilde X\to (\C^n,0)$ be a toric 
log-resolution of $M$ in the sense of Section ~\ref{toric}. 
Consider an entry $R_\tau=c_\tau R_\alpha$ of $R$, where $c_\tau\neq 0$, cf.\
Section ~\ref{zlatan}. Note that $z^{\alpha-\1}R_\alpha\neq 0$, where $\1=(1,\ldots, 1)$. It follows that 
$z^{\alpha-\1}R_\tau\neq
0$, and thus $z^{\alpha-\1}R\neq
0$, which by the duality principle implies that $z^{\alpha-\1}\notin
M$. 
Since
$M$ is integrally closed there is a Rees divisor $D_\tau$, that we may
assume equals $D_1$, of
$M=\overline M$ such that
\begin{equation}\label{vanessa}
\ord_{D_1}(z^{\alpha-\1})<\ord_{D_1}(M),
\end{equation}  see Section
~\ref{svara}.

Since $M$ is monomial, 
$\ord_{D_1}$ is a monomial valuation of the form $\ord_\rho$, where $\rho=(\rho_1,\ldots,
\rho_n)$ is the primitive normal vector of 
one of the compact facets of $\np(M)$; in
particular, $\rho_j\in \mathbf N$, see Sections ~\ref{toric} and
~\ref{svara}. 
Let $\gamma_j=\rho_1\cdots\rho_{j-1}\rho_{j+1}\cdots \rho_n$ and choose
$k\in \Np$ such that $\beta_j:=k \gamma_j \geq \alpha_j$ for all $j$. 
Then $\rho$ is the primitive normal vector of the unique compact facet of the Newton polyhedron of
$\m^\beta=(z_1^{\beta_1},\ldots, z_n^{\beta_n})$, so that ${D_1}$ is the unique Rees divisor of
$\m^\beta$, see Example ~\ref{morgon}. 
It follows that $\pi:\widetilde X\to (\C^n, 0)$ is a log-resolution of
$\m^\beta$, see Section ~\ref{toric}.  
Recall from Section ~\ref{juvel} that (the coefficient of) the Bochner-Martinelli residue
current of $(z_1^{\beta_1},\ldots, z_n^{\beta_n})$ equals $R_\beta$, 
defined as in \eqref{sara}. 
Thus in view of Section ~\ref{juvel}, on $\widetilde X$ there is an
$\widetilde R_\beta$ with support on $D_1$ such that $\pi_*\widetilde
R_\beta = R_\beta$ and 
\[
\widetilde R_\beta = 
\omega_\beta \wedge
\dbar \frac{1}{\sigma_1^{n \ord_{D_1}(\m^\beta)}\cdots \sigma_N^{n
    \ord_{D_N}(\m^\beta)}},
\]
where $\omega_\beta$ is smooth.

Let $\widetilde R_\alpha=\pi^* (z^{\beta-\alpha}) \widetilde
R_\beta$. Then $\widetilde R_\alpha$ has support on $D_1$ and by \eqref{pippi2},  
$\pi_*\widetilde R_\alpha=R_\alpha$. Moreover  
\begin{equation}\label{black}
\widetilde R_\alpha \wedge \pi^*dz = 
\omega \wedge 
\dbar 
\frac{1}{\sigma_1^{a_1}\cdots \sigma_N^{a_N}},
\end{equation}
where $a_i=n\ord_{D_i}(\m^\beta) -
\ord_{D_i}(z^{\beta-\alpha})-\ord_{D_i}(dz)$ and $\omega$ is smooth.  
A direct computation gives that 
$\ord_{D_i}(dz) \geq \ord_{D_i}(z^\1)-1$. 
Since $n\ord_{D_1} (\m^\beta)=\ord_{D_1}(z^\beta)$, see Example
~\ref{morgon}, it follows that 
\[
a_1=\ord_{D_1}(z^\beta)-\ord_{D_1}(z^{\beta-\alpha})-\ord_{D_1}(dz)\leq \ord_{D_1}(z^{\alpha-\1})+1\leq
\ord_{D_1}(M), 
\]
cf.\ \eqref{vanessa}.

That $D=\sum_{i=1}^N r_i D_i$ has simple normal crossings support means that
at $x\in \widetilde X$ we can choose coordinates $s_1,\ldots, s_n$
such that for some $p$, $\pi^{-1}(0)=\{s_1\cdots s_p=0\}$ and for each
$i$ either $x\notin D_i$ or $D_i=\{s_j=0\}$ for some $j$. Thus we may
assume that at $x$, for $i=1,\ldots, p$,
$\sigma_i=s_i\sigma_i'$, where $\sigma_i'$ does not vanish at $x$, and
moreover $\sigma_{p+1},\ldots, \sigma_N$ do not vanish at
$x$. 
Since $a_1\leq r_1=\ord_{D_1}(M)$,  in view of Example
~\ref{kronan}, 
\begin{equation}\label{blanc}
\1_{D_1} \dbar \frac{1}{\sigma_1^{a_1}\cdots \sigma_N^{a_N}}=
\1_{D_1} \sigma_1^{r_1-a_1}\cdots \sigma_N^{r_N-a_N} 
\dbar \frac{1}{\sigma_1^{r_1}\cdots \sigma_N^{r_N}}.
\end{equation}
Let $\eta_\alpha$ be the semi-meromorphic
form 
$\eta_\alpha=\sigma_1^{r_1-a_1}\cdots \sigma_N^{r_N-a_N} \omega$. 
Since $\widetilde R_\alpha$ has support on $D_1$ it follows from
\eqref{black} and \eqref{blanc} that 
\[
\widetilde R_\alpha \wedge \pi^* dz 
= 
\eta_\alpha\wedge 
\dbar \frac{1}{\sigma_1^{r_1}\cdots \sigma_N^{r_N}}. 
\]

Now, let $\eta_\tau=c_\tau\eta_\alpha$ and $\eta=\sum_{\tau\in
  X_n} \eta_\tau ~e_\tau\otimes e_\emptyset^*$. 
 Then $\widetilde R_\tau$ and $\widetilde
R$ are of the desired form. 
\end{proof}

By using the description of residue currents of general, not
necessarily Artinian, monomial ideals in \cite[Section~5]{W} it
should be possible to extend Theorems ~\ref{huvud} and ~\ref{knaochta}
to this setting, although the formulations would become slightly more
complicated. However, the arguments above rely heavily on the explicit
descriptions of the log-resolution of a monomial ideal $M$ and the residue
current $R$ of a cellular resolution of $M$, 
as well as the explicit description of Bochner-Martinelli
residue currents, and it does not seem obvious how to extend them to
non-monomial ideals.

In \cite{LL} Lazarsfeld and Lee proved that multiplier ideals are very
special among integrally closed ideals by proving that the maps
$\varphi_j$ in a free resolution do not vanish to
high order in a certain sense. 
It might happen that in a similar way $R$ has small singularities, in
the sense that it is the pushforward of a current $\widetilde R$ that
satisfies \eqref{underbara}, only for a restricted class of
integrally closed ideals. 

\smallskip 

We finally remark that if the residue current $R$, associated with a
general 
ideal $\a\subset \Ok_0^n$,  
is the pushforward of a current
$\widetilde R$ of the form \eqref{bara}, then, in general, 
 the exponents $a_i$ in \eqref{jul}
have to be (at least) like $n r_i$, where $r_i$ is as in the
introduction. Indeed, assume that for some $\nu\in \mathbf N$, 
$a_i\leq \nu r_i$ for each $i$, and take $g\in
\overline \a^\nu$. Then $\pi^*g$ is divisible by $\sigma$ 
and thus $gR=0$, cf.\ the arguments after \eqref{jul}. 
It follows that 
$\overline \a^\nu\subset \a$. 
The classical Brian\c con-Skoda theorem,
\cite{SB}, asserts that this inclusion holds for $\nu=\min (n,m)$,
where $m$ is the minimum number of generators. This theorem is sharp and therefore in general
the $a_i$ need to be at least like $n r_i$, cf.\ Example
~\ref{bruno}.

\end{document}